\documentclass[oneside,english]{amsart}
\usepackage[T1]{fontenc}
\usepackage[latin9]{inputenc}
\usepackage{color}
\usepackage{verbatim}
\usepackage{mathrsfs}
\usepackage{amstext}
\usepackage{amsthm}
\usepackage{amssymb}
\usepackage{esint}

\makeatletter
\numberwithin{equation}{section}
\numberwithin{figure}{section}
\theoremstyle{plain}
\newtheorem*{thm*}{\protect\theoremname}
\theoremstyle{plain}
\newtheorem*{cor*}{\protect\corollaryname}
\theoremstyle{plain}
\newtheorem{thm}{\protect\theoremname}
\theoremstyle{plain}
\newtheorem{cor}[thm]{\protect\corollaryname}
\theoremstyle{plain}
\newtheorem{lem}[thm]{\protect\lemmaname}
\theoremstyle{remark}
\newtheorem{rem}[thm]{\protect\remarkname}
\theoremstyle{definition}
\newtheorem{example}[thm]{\protect\examplename}

\usepackage[small]{diagrams}

\makeatother

\usepackage{babel}
\providecommand{\corollaryname}{Corollary}
\providecommand{\examplename}{Example}
\providecommand{\lemmaname}{Lemma}
\providecommand{\remarkname}{Remark}
\providecommand{\theoremname}{Theorem}

\begin{document}
\title[Global Lagrangian construction on $2$-manifolds]{On a Global Lagrangian construction for ordinary variational equations
on 2-manifolds}
\author{Zbyn\v{e}k Urban}
\address[Zbyn\v{e}k Urban]{Dpt. of Mathematics, Faculty of Civil Engineering, V\v{S}B-Technical University of Ostrava\\
Ludv\'{i}ka Pod\'{e}\v{s}t\v{e} 1875/17, 708 33 Ostrava,
Czech Republic}
\email[Corresponding author]{zbynek.urban@vsb.cz}
\author{Jana Voln\'a}
\address[Jana Voln\'a]{Dpt. of Mathematics, Faculty of Civil Engineering, V\v{S}B-Technical University of Ostrava\\
Ludv\'{i}ka Pod\'{e}\v{s}t\v{e} 1875/17, 708 33 Ostrava,
Czech Republic}
\email{jana.volna@vsb.cz}
\begin{abstract}
Locally variational systems of differential equations on smooth manifolds,
having certain de Rham cohomology group trivial, automatically possess
a global Lagrangian. This important result due to Takens is, however,
of sheaf-theoretic nature. A new constructive method of finding a
global Lagrangian for second-order ODEs on 2-manifolds is described
on the basis of solvability of exactness equation for Lepage 2-forms,
and the top-cohomology theorems. Examples from geometry and mechanics
are discussed.
\end{abstract}

\thanks{ZU appreciates support of the Visegrad grant No. 51810810 at the University
of Pre\v{s}ov.}
\maketitle

\section{Introduction}

In this paper, we address the problem of finding a \emph{concrete}
\textit{global} Lagrangian for \textit{variational} second-order ODEs on smooth fibered manifolds. The existence
of a~global variational principle for given equations is essentially
influenced by topology of the underlying space: for \emph{ordinary}
variational equations of arbitrary order it depends on the \textit{second
de Rham cohomology} group $H_{\mathrm{dR}}^{2}Y$ of the underlying
fibered manifold $Y$. The following theorem belongs to important\emph{
}global\emph{ }results achieved within the calculus of variations
on smooth manifolds. 
\begin{thm*}[Takens, 1979]
\label{thm:Takens} Each locally variational source equation is globally
variational provided $H_{\mathrm{dR}}^{n+1}\left(Y;\mathbb{R}\right)=0.$
\end{thm*}
In \cite{Takens}, Takens obtained this result within the framework
of a\emph{~variational bicomplex} theory over sheaves of differential
forms on \emph{infinite} jet prolongations of fibered manifolds over
general \emph{$n$}-dimensional bases. An analogous global result
to this theorem was also obtained by Vinogradov \cite{Vin1,Vin2}
(see \cite{Vin3} for more detail exposition as a~part of $\mathscr{C}$-spectral
sequences using Spencer cohomology). Variational bicomplex theories
have been developed since late seventieth by many authors with the
aim to study a~complex, analogous to the de Rham complex, where the
Euler\textendash Lagrange mapping is included as one of its morphisms
in an exact sequence; see Anderson and Duchamp \cite{Anderson}, Dedecker
and Tulczyjew \cite{Dedecker}, and Tulczyjew \cite{Tulczyjew}. 

A~different approach to similar ideas was developed in the variational
sequence theory by Krupka \cite{Krupka-VarSeq,Krupka-VarSeqMech},
who considered the quotient sequence of the de Rham sequence over
\emph{finite-order} jet prolongations of fibered manifolds with respect
to its \emph{contact subsequence}. Thus, a basic concept of the calculus
of variations, the \emph{Euler\textendash Lagrange mapping}, can be
constructed (for variational functionals on fibered spaces) as the
quotient mapping of the exterior derivative operator $d$, acting
on differential forms, by the restriction of $d$ to the so-called
\textit{contact forms}. The main global result of Krupka \cite{Krupka-VarSeq} reads: \textit{The variational sequence of order $r$ over $Y$ is an acyclic resolution of the constant sheaf $\mathbb{R}_{Y}$ over $Y$.} From this theorem and the well-known abstract de Rham theorem, we immediately
obtain the following corollary: \textit{Let $\varepsilon$ be a locally variational source
form on $J^{r}Y$. If $H_{\mathrm{dR}}^{n+1}Y$ is trivial, then $\varepsilon$
is also globally variational.} This is a~finite-order analogue of the result due to Takens \cite{Takens}.

One should notice, however, that the cohomology conditions in the
variational sequence and the variational bicomplex theory have \emph{different}
meaning. The relationship of these two theories can be found in Pommaret
\cite{Pommaret}, Vitolo \cite{Vitolo}, and Krupka \emph{et. al.}
\cite{KruMorUrbVol}.

A common feature of the previously mentioned works \cite{Anderson,Dedecker,Krupka-VarSeq,Krupka-VarSeqMech,Takens,Tulczyjew,Vin1,Vin2,Vin3,Vitolo}
is an absence of a~\emph{concrete global variational principle},
whose existence is guaranteed by cohomology conditions. To the authors'
knowledge, there is \emph{no general} method how to construct a~global
Lagrangian for locally variational equations. Note that simple examples
show that the well-known Vainberg\textendash Tonti formula (cf. Tonti
\cite{Tonti}, Krupka \cite{Krupka-Lepage}) fails to produce Lagrangians
which are defined globally.

In the framework of global \textcolor{black}{variational} theory on
finite-order jet prolongations of fibered manifolds (cf. Krupka \cite{Book},
and references therein), the formulation of our main problem is the
following: let $\varepsilon$ be a~given \emph{locally variational}
\textit{source form} (also known as a~\textit{dynamical form} in Lagrangian mechanics)
on the second jet prolongation $J^{2}Y$ of a fibered manifold $Y$ over a~$1$-dimensional base.
Then we search for a~(concrete) global Lagrangian $\lambda$, which is a~horizontal differential
$1$-form defined (globally) on $J^{1}Y$, such that $\varepsilon$
coincides with the image of $\lambda$ in the \emph{Euler\textendash Lagrange
mapping}, that is, with the \emph{Euler\textendash Lagrange form}
associated with $\lambda$.

Section 2 is devoted to basic facts of \emph{second-order} variational
ODEs in accordance with the general theory
(cf. Krupkov\'a and Prince \cite{Krupkova-Prince,KruPri}, Krupkov\'a
\cite{Krupkova}), including necessary and sufficient conditions for
local variationality, namely the \emph{Helmholtz conditions}. The geometry of second-order PDEs has been studied recently by Saunders, Rossi and Prince \cite{SRP}. For
local variational principles based on Lepage forms, see Brajer\v{c}\'ik
and Krupka \cite{Brajercik}. 

Recall that a Lepage form represents
a~far-going generalization of a 1-form, introduced by E.~Cartan
within the framework of the calculus of variations (see Krupka \cite{Krupka-Amerika,Krupka-Lepage}).
Roughly speaking, a~Lepage form gives a~geometric description of
the associated variational functional: its variations, extremals and
invariance are characterized by means of the geometric operations
as the exterior derivative, and the Lie derivative of differential forms.
The meaning of Lepage forms for the calculus of variations and their
basic properties have recently been summarized by Krupka, Krupková
and Saunders \cite{KKS}.

In Section 3, we develop the main idea of this paper: to tackle the
problem on the basis of solvability of the \emph{global exactness}
equation for the~\emph{Lepage equivalent }$\alpha_{\varepsilon}$
of a~source form $\varepsilon$. Globally defined $2$-form
$\alpha_{\varepsilon}$ represents an example of a~\emph{Lepage $2$-form}
in Lagrangian mechanics (see Krupkov\'{a} \cite{Krupkova1986,Krupkova}),
and it satisfies the condition $\alpha_{\varepsilon}=d\Theta_{\lambda}$,
which is a~subject of solution with respect to unknown $\lambda$, where $\Theta_{\lambda}$  is the well-known \emph{Cartan form}. Recall that $\Theta_{\lambda}$ depends on the choice of a~Lagrangian $\lambda$
whereas $d\Theta_{\lambda}$ \emph{does not}. As a~result, we reduce the global exactness of Lepage equivalent
$\alpha_{\varepsilon}$ of $\varepsilon$ to global exactness of a~certain
$2$-form $\omega$, defined on the underlying fibered manifold $Y$.

In Section 4, we apply the standard de Rham \emph{top}-cohomology
theory (see Lee \cite{Lee}) to solve global exactness of the differential
$2$-form $\omega$, and combine it with results obtained in Section
3, see Theorem \ref{thm:Top-Var}. Indeed, this step can be proceeded
on smooth manifolds of dimension \emph{two} only, and our general
global Lagrangian construction for locally variational source forms
is therefore restricted to $2$-manifolds with \emph{trivial} the
second de Rham cohomology group of $Y$. We point out that the equation $\omega=d\eta$ need not have a~global solution, and even if solvability of this equation is assured, \textit{no general} construction of its solution is known on smooth $m$-dimensional manifolds. This circumstance makes our problem difficult in general. Note that topic of this paper
is closely related to the \emph{inverse problem of the calculus of
variations}, discussed by most of the authors mentioned above, and
solved by J.~Douglas in his seminal paper (1941) for systems of \emph{two}
ordinary equations of \emph{two} dependent variables. 

Section 5 contains two examples of mechanical systems, 
namely, the \textit{kinetic energy} on the \textit{open M\"obius strip}, and a~\textit{gyroscopic
type} system on the \textit{punctured torus} in the Euclidean space $\mathbb{R}^{3}$, where the corresponding
global variational principles are discussed. We emphasize, however, that the theory
can\emph{ not} be applied to general $m$-manifolds; e.g. for $m=3$,
$H_{\mathrm{dR}}^{2}S^{3}$ of the $3$-sphere $S^{3}$ is \textit{trivial},
nevertheless we can proceed only under additional requirements (cf.
Corollary \ref{cor:Simple}).

Notation and underlying geometric structures are coherent with our
recent work Krupka, Urban, and Voln\'a \cite{KUV}, where general
higher-order formulas can be found. Throughout, we consider fibered
manifolds the Cartesian products $Y=\mathbb{R}\times M$ over the
\emph{real line} $\mathbb{R}$ and projection $\pi$, where $M$ is
a~smooth connected $2$-manifold. Thus, the jet prolongations $J^{1}Y$
and $J^{2}Y$ of $Y$ can be canonically identified with the product
$\mathbb{R}\times T^{1}M$ and $\mathbb{R}\times T^{2}M$, respectively,
where $T^{1}M$ is the tangent bundle of $M$, and $T^{2}M$ denotes
the manifold of \emph{second-order velocities} over $M$. Recall that
elements of $T^{2}M$ are $2$-jets $J_{0}^{2}\zeta\in J^{2}\left(\mathbb{R},M\right)$
with origin $0\in\mathbb{R}$ and target $\zeta(0)\in M$. The jet
prolongations are considered with its natural fibered manifold structure:
if $\left(V,\psi\right)$, $\psi=\left(t,x,y\right)$, is a~fibered
chart on $\mathbb{R}\times M$, the associated chart on $J^{2}Y$
(respectively, $J^{1}Y$) reads $\left(V^{2},\psi^{2}\right)$, $\psi^{2}=\left(t,x,y,\dot{x},\dot{y},\ddot{x},\ddot{y}\right)$
(respectively, $\left(V^{1},\psi^{1}\right)$, $\psi^{1}=\left(t,x,y,\dot{x},\dot{y}\right)$),
where $V^{2}$ (respectively, $V^{1}$) is the preimage of $V$ in
the canonical jet projection $\pi^{2,0}:J^{2}Y\rightarrow Y$ (respectively,
$\pi^{1,0}:J^{1}Y\rightarrow Y$).

Let $W$ be an open set in $Y$, and $\Omega^{1}W$ the exterior algebra
of differential forms on $W^{1}$. By means of charts, we put $hdt=dt$,
$hdx=\dot{x}dt$, $hdy=\dot{y}dt$, $hd\dot{x}=\ddot{x}dt$, $hd\dot{y}=\ddot{y}dt$,
and for any function $f:W^{1}\rightarrow\mathbb{R}$, $hf=f\circ\pi^{2,1}$,
where $\pi^{2,1}$ denotes the canonical jet projection $J^{2}Y\rightarrow J^{1}Y$.
These formulas define a~global homomorphism of exterior algebras
$h:\Omega^{1}W\rightarrow\Omega^{2}W$, called the $\pi$-\emph{horizontalization}.
A~$1$-form $\rho$ on $W^{1}$ is called \emph{contact}, if $h\rho=0$.
In a~fibered chart $\left(V,\psi\right)$, $\psi=\left(t,x,y\right)$,
on $W$, every contact $1$-form $\rho$ has an expression $\rho=A_{x}\omega^{x}+A_{y}\omega^{y}$,
for some functions $A_{x},A_{y}:V^{1}\rightarrow\mathbb{R},$ where
$\omega^{x}=dx-\dot{x}dt$, $\omega^{y}=dy-\dot{y}dt$. For any differential
$1$-form $\rho$ on $W^{1}$, the pull-back $\left(\pi^{2,1}\right)^{*}\rho$
has a~unique decomposition $\left(\pi^{2,1}\right)^{*}\rho=h\rho+p\rho$,
where $h\rho$, resp. $p\rho$, is $\pi^{2}$-\emph{horizontal} (respectively,
\emph{contact}) $1$-form on $W^{2}$. This decomposition can be directly
generalized to arbitrary $k$-forms. For $k=2$, if $\rho$ is a~$2$-form
on $W^{1}$, then we get $\left(\pi^{2,1}\right)^{*}\rho=p_{1}\rho+p_{2}\rho$,
where $p_{1}\rho$ (resp. $p_{2}\rho$) is the $1$-\emph{contact}
(respectively, $2$-\emph{contact}) component of $\rho$, spanned
by $\omega^{x}\wedge dt$, $\omega^{y}\wedge dt$ (respectively, $\omega^{x}\wedge\omega^{y}$).
Analogously, we employ these concepts on $W^{2}$.

The results of this work can be generalized to higher-order variational
differential equations by means of similar methods.

\section{Second-order ordinary variational equations}

Let $\varepsilon$ be a \emph{source form} on $\mathbb{R}\times T^{2}M$.
By definition, $\varepsilon$ is a $1$-contact and $\pi^{2,0}$-horizontal
2-form. In a chart $\left(V,\psi\right)$, $\psi=\left(t,x,y\right)$,
on $\mathbb{R}\times M$, $\varepsilon$ is expressed as
\begin{align}
\varepsilon & =\left(\varepsilon_{x}\omega^{x}+\varepsilon_{y}\omega^{y}\right)\wedge dt,\label{eq:Source}
\end{align}
where the coefficients $\varepsilon_{x}$, $\varepsilon_{y}$ are
differentiable functions on $V^{2}$ , and $\omega^{x}=dx-\dot{x}dt$,
$\omega^{y}=dy-\dot{y}dt$ are contact $1$-forms on $V^{1}$. To
simplify further considerations, but without loss of generality, we
suppose that $\varepsilon_{x}$, $\varepsilon_{y}$ do \textit{not}
depend on the time variable $t$ explicitly, that is $\varepsilon_{x}$,
$\varepsilon_{y}$ are functions of $x$, $y$, $\dot{x}$, $\dot{y}$,
$\ddot{x}$, $\ddot{y}$ only. Source form \eqref{eq:Source} associates
a~system of two second-order differential equations
\begin{equation}
\varepsilon_{x}\left(x,y,\dot{x},\dot{y},\ddot{x},\ddot{y}\right)=0,\quad\varepsilon_{y}\left(x,y,\dot{x},\dot{y},\ddot{x},\ddot{y}\right)=0,\label{eq:SourceSystem}
\end{equation}
for unknown differentiable curves $\zeta$ in $M$, $t\rightarrow\zeta(t)=\left(x\circ\zeta(t),y\circ\zeta(t)\right)$,
defined on an open interval of $\mathbb{R}$. 

$2$-form \eqref{eq:Source} (or system \eqref{eq:SourceSystem})
is called \emph{locally variational}, if there exists a~real-valued
function $\mathscr{L}$ on a chart neighborhood $V^{1}$, $\mathscr{L}=\mathscr{L}\left(t,x,y,\dot{x},\dot{y}\right)$,
such that \eqref{eq:SourceSystem} coincide with the \emph{Euler\textendash Lagrange
equations}, i.e. $\varepsilon_{x}=E_{x}(\mathscr{L})$ and $\varepsilon_{y}=E_{y}(\mathscr{L})$
are the \emph{Euler\textendash Lagrange expressions}, associated with
$\mathscr{L}$,
\begin{equation}
E_{x}(\mathscr{L})=\frac{\partial\mathscr{L}}{\partial x}-\frac{d}{dt}\frac{\partial\mathscr{L}}{\partial\dot{x}},\quad E_{y}(\mathscr{L})=\frac{\partial\mathscr{L}}{\partial y}-\frac{d}{dt}\frac{\partial\mathscr{L}}{\partial\dot{y}}.\label{eq:E-L-expressions}
\end{equation}

A~\emph{Lagrangian} of order $1$ for $Y$ is by definition a~$\pi^{1}$-horizontal
$1$-form $\lambda$ on $W^{1}\subset\mathbb{R}\times T^{1}M$. In
a~fibered chart, $\lambda=\mathscr{L}dt$, where $\mathscr{L}:V^{1}\rightarrow\mathbb{R}$
is a~real-valued function called the (local) \emph{Lagrange function}
associated with $\lambda$. Every Lagrangian $\lambda$ associates
a~source form $E_{\lambda}$, locally expressed by
\[
E_{\lambda}=E_{x}(\mathscr{L})\omega^{x}\wedge dt+E_{y}(\mathscr{L})\omega^{y}\wedge dt.
\]
$E_{\lambda}$ is called the \emph{Euler-Lagrange form}, associated
with $\lambda$. Thus, locally variational forms belong to \emph{image}
of the \emph{Euler\textendash Lagrange mapping} $\lambda\rightarrow E_{\lambda}$.
Note that a~Lagrangian is a~representative of class of $1$-forms,
whereas a~source form is a~representative of class of $2$-forms
in the (quotient) variational sequence over $W\subset Y$; for further
remarks see Krupka \cite{Krupka-VarSeqMech}.

From the definition, it is easy to observe that the coefficients of
locally variational form \eqref{eq:Source} coincide with the Euler\textendash Lagrange
expressions of a Lagrange function with respect to \emph{every} chart.
Nevertheless, such local Lagrange functions, defined on chart neighborhoods
in $\mathbb{R}\times T^{1}M$, need \emph{not} define a (global) function
on $\mathbb{R}\times T^{1}M$. If there exists a Lagrange function
$\mathscr{L}$ for $\varepsilon$ defined on $\mathbb{R}\times T^{1}M$,
then $\varepsilon$ is called \emph{globally variational}.

In the following theorem, we give necessary and sufficient conditions
for locally variational source forms.
\begin{thm}
\label{thm:LocalVariationality}Let $\varepsilon$ be a source form
on $\mathbb{R}\times T^{2}M$, locally expressed by \eqref{eq:Source}
with respect to a chart $\left(V,\psi\right)$, $\psi=\left(t,x,y\right)$,
on $\mathbb{R}\times M$. The following conditions are equivalent: 

\emph{(a)} $\varepsilon$ is locally variational.

\emph{(b)} The functions $\varepsilon_{x}$, $\varepsilon_{y}$ satisfy
identically the system
\begin{align}
 & \frac{\partial\varepsilon_{x}}{\partial\ddot{y}}-\frac{\partial\varepsilon_{y}}{\partial\ddot{x}}=0,\quad\frac{\partial\varepsilon_{x}}{\partial\dot{x}}-\frac{d}{dt}\frac{\partial\varepsilon_{x}}{\partial\ddot{x}}=0,\quad\frac{\partial\varepsilon_{y}}{\partial\dot{y}}-\frac{d}{dt}\frac{\partial\varepsilon_{y}}{\partial\ddot{y}}=0,\nonumber \\
 & \frac{\partial\varepsilon_{x}}{\partial\dot{y}}+\frac{\partial\varepsilon_{y}}{\partial\dot{x}}-\frac{d}{dt}\left(\frac{\partial\varepsilon_{x}}{\partial\ddot{y}}+\frac{\partial\varepsilon_{y}}{\partial\ddot{x}}\right)=0,\label{eq:Helm}\\
 & \frac{\partial\varepsilon_{x}}{\partial y}-\frac{\partial\varepsilon_{y}}{\partial x}-\frac{1}{2}\frac{d}{dt}\left(\frac{\partial\varepsilon_{x}}{\partial\dot{y}}-\frac{\partial\varepsilon_{y}}{\partial\dot{x}}\right)=0.\nonumber 
\end{align}

\emph{(c)} The functions $\varepsilon_{x}$, $\varepsilon_{y}$ are
of the form
\begin{align}
 & \varepsilon_{x}=A_{x}+B_{xx}\ddot{x}+B_{xy}\ddot{y},\quad\varepsilon_{y}=A_{y}+B_{yx}\ddot{x}+B_{yy}\ddot{y},\label{eq:LinearEpsilon}
\end{align}
where the functions $A_{x},A_{y},B_{xx},B_{xy},B_{yx},B_{yy}$ depend
on $x,y,\dot{x},\dot{y}$ only, and satisfy the following system identically,
\begin{align}
 & B_{xy}=B_{yx},\quad\frac{\partial B_{xx}}{\partial\dot{y}}=\frac{\partial B_{xy}}{\partial\dot{x}},\quad\frac{\partial B_{yy}}{\partial\dot{x}}=\frac{\partial B_{xy}}{\partial\dot{y}},\nonumber \\
 & \frac{\partial A_{x}}{\partial\dot{x}}-\frac{\partial B_{xx}}{\partial x}\dot{x}-\frac{\partial B_{xx}}{\partial y}\dot{y}=0,\quad\frac{\partial A_{y}}{\partial\dot{y}}-\frac{\partial B_{yy}}{\partial x}\dot{x}-\frac{\partial B_{yy}}{\partial y}\dot{y}=0,\label{eq:HelmAB}\\
 & \frac{\partial A_{x}}{\partial\dot{y}}+\frac{\partial A_{y}}{\partial\dot{x}}-2\frac{\partial B_{xy}}{\partial x}\dot{x}-2\frac{\partial B_{xy}}{\partial y}\dot{y}=0,\nonumber \\
 & \frac{\partial A_{x}}{\partial y}-\frac{\partial A_{y}}{\partial x}-\frac{1}{2}\frac{\partial}{\partial x}\left(\frac{\partial A_{x}}{\partial\dot{y}}-\frac{\partial A_{y}}{\partial\dot{x}}\right)\dot{x}-\frac{1}{2}\frac{\partial}{\partial y}\left(\frac{\partial A_{x}}{\partial\dot{y}}-\frac{\partial A_{y}}{\partial\dot{x}}\right)\dot{y}=0.\nonumber 
\end{align}

\emph{(d)} The function
\begin{equation}
\mathscr{L}=\mathscr{L}_{T}-\frac{d}{dt}\left(x\int_{0}^{1}C_{x}\left(sx,sy,s\dot{x},s\dot{y}\right)ds+y\int_{0}^{1}C_{y}\left(sx,sy,s\dot{x},s\dot{y}\right)ds\right),\label{eq:VT-1stOrder}
\end{equation}
where functions $C_{x}$, $C_{y}$ are given by conditions \eqref{eq:HelmAB}
as $B_{xy}=\partial C_{x}/\partial\dot{y}=\partial C_{y}/\partial\dot{x}$,
and
\begin{equation}
\mathscr{L}_{T}=x\int_{0}^{1}\varepsilon_{x}\left(sx,sy,s\dot{x},s\dot{y},s\ddot{x},s\ddot{y}\right)ds+y\int_{0}^{1}\varepsilon_{y}\left(sx,sy,s\dot{x},s\dot{y},s\ddot{x},s\ddot{y}\right)ds,\label{eq:VT-Lagrangian}
\end{equation}
 is a Lagrange function for $\varepsilon$ defined on $V^{1}$.

\emph{(e)} To every point of $\mathbb{R}\times T^{2}M$ there is a
neighborhood $W$ and a $2$-contact $2$-form $F_{W}$ on $W$ such
that the form $\alpha_{W}=\varepsilon|_{W}+F_{W}$ is closed. 

\emph{(f)} There exists a closed $2$-form $\alpha_{\varepsilon}$
on $\mathbb{R}\times T^{1}M$ such that $\varepsilon=p_{1}\alpha_{\varepsilon}$.
If $\alpha_{\varepsilon}$ exists, it is unique and it has a chart
expression given by
\begin{align}
\alpha_{\varepsilon} & =\left(\varepsilon_{x}\omega^{x}+\varepsilon_{y}\omega^{y}\right)\wedge dt+\frac{1}{2}\left(\frac{\partial A_{x}}{\partial\dot{y}}-\frac{\partial A_{y}}{\partial\dot{x}}\right)\omega^{x}\wedge\omega^{y}\label{eq:Alfa}\\
 & +B_{xx}\omega^{x}\wedge\dot{\omega}^{x}+B_{xy}\left(\omega^{x}\wedge\dot{\omega}^{y}+\omega^{y}\wedge\dot{\omega}^{x}\right)+B_{yy}\omega^{y}\wedge\dot{\omega}^{y},\nonumber 
\end{align}
where $\omega^{x}=dx-\dot{x}dt$, $\omega^{y}=dy-\dot{y}dt$, $\dot{\omega}^{x}=d\dot{x}-\ddot{x}dt$,
$\dot{\omega}^{y}=d\dot{y}-\ddot{y}dt$, are contact $1$-forms on
$V^{2}$.
\end{thm}

The identities expressed by Theorem \ref{thm:LocalVariationality},
(b), or equivalently (c), are the well-known \emph{Helmholtz conditions}
of local variationality (cf. Krupkov\'a and Prince, and references
therein). Formula \eqref{eq:VT-Lagrangian} yields the \emph{Vainberg\textendash Tonti
Lagrange function} for a~locally variational source form $\varepsilon$
(see Tonti \cite{Tonti}), which is defined on $V^{2}$ and can always
be reduced to first-order Lagrange function \eqref{eq:VT-1stOrder}
on $V^{1}$. Note that the Euler\textendash Lagrange form associated
with Lagrangian \eqref{eq:VT-1stOrder} coincides with source form
$\varepsilon$, provided the Helmholtz conditions are satisfied.

A~relationship between locally variational source forms and \emph{closed}
forms was studied by Krupka \cite{Krupka-Lepage}, and is given by
Theorem \ref{thm:LocalVariationality}, (e). Theorem \ref{thm:LocalVariationality},
(f), represents\emph{ global} generalization of condition (e) due
to Krupkov\'a \cite{Krupkova1986}. A~straightforward coordinate
transformations applied to formula \eqref{eq:Alfa} verify that $\alpha$
defines a \emph{global} form on $\mathbb{R}\times T^{1}M$, and it
represents an example of a~\emph{Lepage $2$-form} in mechanics (see
also Krupkov\'a and Prince \cite{KruPri,Krupkova-Prince}). $\alpha_{\varepsilon}$
is called a \emph{Lepage equivalent} of locally variational source
form $\varepsilon$.

A~$1$-form $\vartheta$ on $W^{1}\subset\mathbb{R}\times T^{1}M$
is called a~\emph{Lepage form} (of first-order), if the contraction
$i_{\xi}d\vartheta$ is contact $1$-form for every $\pi^{1,0}$-vertical
vector field $\xi$ on $W^{1}$. In addition, if $h\vartheta=\lambda$
for a~Lagrangian $\lambda$ on $W^{1}$, $\vartheta$ is called the
\emph{Lepage equivalent} of $\lambda$. The concepts of a~Lepage
form and the~Lepage equivalent of a~Lagrangian are introduced for
finite-order jet prolongations of fibered manifolds over $n$-dimensional
basis (see Krupka \cite{Krupka-Lepage}), and also the Grassmann fibrations
(see Urban and Krupka \cite{UK}); for $n=1$ (fibered mechanics),
the Lepage equivalent of a~Lagrangian is unique. The following theorem
recalls the well-known \emph{Cartan form} $\Theta_{\lambda}$ of a
\emph{first-order} Lagrangian $\lambda$, which represents an example
of a~Lepage form. For variational principles in mechanics based on
the Cartan form and its generalizations, see Krupka, Krupkov\'a and
Saunders \cite{KKS}.
\begin{thm}
\label{thm:Cartan} Every first-order Lagrangian $\lambda\in\Omega_{1,X}^{1}W$
has a unique Lepage equivalent $\Theta_{\lambda}$. If $\lambda$
has a chart expression $\lambda=\mathscr{L}dt$, then
\begin{equation}
\Theta_{\lambda}=\mathscr{L}dt+\frac{\partial\mathscr{L}}{\partial\dot{x}}\omega^{x}+\frac{\partial\mathscr{L}}{\partial\dot{y}}\omega^{y},\label{eq:CartanForm}
\end{equation}
and
\begin{equation}
p_{1}d\Theta_{\lambda}=E_{\lambda},\label{eq:p1dTheta}
\end{equation}
where 
\begin{equation}
E_{\lambda}=\left(E_{x}(\mathscr{L})\omega^{x}+E_{y}(\mathscr{L})\omega^{y}\right)\wedge dt.\label{eq:EL-form}
\end{equation}
\end{thm}

Combining Theorem \ref{thm:LocalVariationality}, (f), with the Lepage
equivalent property \eqref{eq:p1dTheta}, we get a~straightforward
observation for globally variational source forms. 
\begin{cor}
Let $\varepsilon$ be a source form on $\mathbb{R}\times T^{2}M$,
which is locally variational. If the equation
\begin{equation}
\alpha_{\varepsilon}=d\Theta_{\lambda}\label{eq:ExteriorEquation}
\end{equation}
has a (global) solution $\lambda\in\Omega_{1,X}^{1}\left(\mathbb{R}\times T^{1}M\right)$,
then $\varepsilon$ is also globally variational, and vice versa.
\end{cor}

\begin{proof}
Theorem \ref{thm:LocalVariationality}, (f), assures a unique $2$-form
$\alpha_{\varepsilon}$ on $\mathbb{R}\times T^{1}M$ which is closed
and satisfies $\varepsilon=p_{1}\alpha_{\varepsilon}$. Applying the
operator $p_{1}$ onto \eqref{eq:ExteriorEquation}, we get $\varepsilon=E_{\lambda}$
for some $\lambda\in\Omega_{1,X}^{1}\left(\mathbb{R}\times T^{1}M\right)$.
\end{proof}

\section{The exactness equation for Lepage 2-form}

Let $\varepsilon$ be a locally variational source form defined on
$\mathbb{R}\times T^{2}M$, and consider the Lepage equivalent $\alpha_{\varepsilon}$
of $\varepsilon$ (Theorem \ref{thm:LocalVariationality}, (f)). The
Poincar\'e lemma implies that $\alpha_{\varepsilon}$ is locally
exact. We observe that the main problem of finding a~global Lagrangian
for $\varepsilon$ is closely related with \emph{global exactness}
\emph{of Lepage 2-form} $\alpha_{\varepsilon}$ or, in other words,
with finding a solution $\mu$ defined on $\mathbb{R}\times T^{1}M$
of the equation 
\begin{equation}
\alpha_{\varepsilon}=d\mu.\label{eq:AlfaExactness}
\end{equation}
A~solution $\mu$ of \eqref{eq:AlfaExactness} produces Lagrangian
$h\mu$\textcolor{black}{, which is }\emph{equivalent} to $\lambda=\mathscr{L}dt$
given by equation \eqref{eq:ExteriorEquation}. Indeed, if $\mu$
solves \eqref{eq:AlfaExactness} and $\alpha_{\varepsilon}=d\Theta_{\lambda}$,
then $\mu=\Theta_{\lambda}+df$, hence we get $h\mu=\lambda+h\left(df\right)=\left(\mathscr{L}+df/dt\right)dt$,
which differs from $\lambda$ by means of total derivative of a function.

Recall that equation \eqref{eq:AlfaExactness} need \textit{not} have a~global solution on $T^1M$ and, moreover, if a solution exists, there is \emph{no} construction of this solution on general $m$-dimensional manifolds; see Remark \ref{rem:Lee}.

The next lemma allows \emph{global} canonical decomposition of $\alpha_{\varepsilon}$
into closed forms.
\begin{lem}
\label{lem:AlfaClosed}Let $\alpha_{\varepsilon}$ be the Lepage equivalent
of a~locally variational source form $\varepsilon$ on $\mathbb{R}\times T^{2}M$.
Then there is a~unique decomposition of $\alpha_{\varepsilon}$ on
$\mathbb{R}\times T^{1}M$, 

\begin{equation}
\alpha_{\varepsilon}=\alpha_{0}\wedge dt+\alpha',\label{eq:AlfaDecomposition}
\end{equation}
where $\alpha_{0}$ and $\alpha'$ are closed forms defined on $T^{1}M$,
and $\alpha'$ does not contain $dt$. In a~fibered chart $\left(V,\psi\right)$,
$\psi=\left(t,x,y\right)$, on $\mathbb{R}\times M$, we have

\begin{align}
\alpha_{0} & =\left(A_{x}-\frac{1}{2}\left(\frac{\partial A_{x}}{\partial\dot{y}}-\frac{\partial A_{y}}{\partial\dot{x}}\right)\dot{y}\right)dx+\left(A_{y}-\frac{1}{2}\left(\frac{\partial A_{y}}{\partial\dot{x}}-\frac{\partial A_{x}}{\partial\dot{y}}\right)\dot{x}\right)dy\nonumber \\
 & +\left(B_{xx}\dot{x}+B_{xy}\dot{y}\right)d\dot{x}+\left(B_{xy}\dot{x}+B_{yy}\dot{y}\right)d\dot{y},\label{eq:Alfa0}
\end{align}
and
\begin{align}
\alpha' & =\frac{1}{2}\left(\frac{\partial A_{x}}{\partial\dot{y}}-\frac{\partial A_{y}}{\partial\dot{x}}\right)dx\wedge dy\label{eq:AlfaPrime}\\
 & +\left(B_{xx}dx+B_{xy}dy\right)\wedge d\dot{x}+\left(B_{xy}dx+B_{yy}dy\right)\wedge d\dot{y}.\nonumber 
\end{align}
\end{lem}

\begin{proof}
Since the Lepage $2$-form $\alpha_{\varepsilon}$ is closed, it is
sufficient to show that also $\alpha'$ is closed. This can be, however,
directly verified in a~fibered chart employing the Helmholtz conditions
\eqref{eq:HelmAB}. Indeed, from \eqref{eq:AlfaPrime} we have a~chart
expression
\begin{align*}
d\alpha' & =\left(\frac{\partial B_{xy}}{\partial x}-\frac{\partial B_{xx}}{\partial y}+\frac{1}{2}\frac{\partial}{\partial\dot{x}}\left(\frac{\partial A_{x}}{\partial\dot{y}}-\frac{\partial A_{y}}{\partial\dot{x}}\right)\right)dx\wedge dy\wedge d\dot{x}\\
 & +\left(\frac{\partial B_{yy}}{\partial x}-\frac{\partial B_{xy}}{\partial y}+\frac{1}{2}\frac{\partial}{\partial\dot{y}}\left(\frac{\partial A_{x}}{\partial\dot{y}}-\frac{\partial A_{y}}{\partial\dot{x}}\right)\right)dx\wedge dy\wedge d\dot{y}\\
 & +\left(\frac{\partial B_{xx}}{\partial\dot{y}}-\frac{\partial B_{xy}}{\partial\dot{x}}\right)dx\wedge d\dot{x}\wedge d\dot{y}+\left(\frac{\partial B_{xy}}{\partial\dot{y}}-\frac{\partial B_{yy}}{\partial\dot{x}}\right)dy\wedge d\dot{x}\wedge d\dot{y},
\end{align*}
where the last two summands vanish using
\begin{equation}
\frac{\partial B_{xx}}{\partial\dot{y}}=\frac{\partial B_{xy}}{\partial\dot{x}},\quad\frac{\partial B_{yy}}{\partial\dot{x}}=\frac{\partial B_{xy}}{\partial\dot{y}},\label{eq:Helmholtz-1}
\end{equation}
and from the Helmholtz conditions \eqref{eq:HelmAB} it is easy to
see that the following identities hold
\begin{align}
 & \frac{\partial B_{xy}}{\partial x}-\frac{\partial B_{xx}}{\partial y}+\frac{1}{2}\frac{\partial}{\partial\dot{x}}\left(\frac{\partial A_{x}}{\partial\dot{y}}-\frac{\partial A_{y}}{\partial\dot{x}}\right)=0,\label{eq:Helmholtz-2}\\
 & \frac{\partial B_{yy}}{\partial x}-\frac{\partial B_{xy}}{\partial y}+\frac{1}{2}\frac{\partial}{\partial\dot{y}}\left(\frac{\partial A_{x}}{\partial\dot{y}}-\frac{\partial A_{y}}{\partial\dot{x}}\right)=0.\nonumber 
\end{align}
Hence $d\alpha'=0$, as required. It remains to show that both local
forms $\alpha_{0}$ \eqref{eq:Alfa0} and $\alpha'$ \eqref{eq:AlfaPrime}
define global forms on $\mathbb{R}\times T^{1}M$. Since $\varepsilon$
\eqref{eq:Source} is defined on $\mathbb{R}\times T^{1}M$, we get
for an arbitrary~chart transformation $x=x(\bar{x},\bar{y})$, $y=y(\bar{x},\bar{y})$,
on $M$, the relations
\begin{align}
 & \frac{\partial A_{x}}{\partial\dot{y}}-\frac{\partial A_{y}}{\partial\dot{x}}=\left(\frac{\partial A_{\bar{x}}}{\partial\dot{\bar{y}}}-\frac{\partial A_{\bar{y}}}{\partial\dot{\bar{x}}}\right)\left(\frac{\partial\bar{y}}{\partial y}\frac{\partial\bar{x}}{\partial x}-\frac{\partial\bar{x}}{\partial y}\frac{\partial\bar{y}}{\partial x}\right)\nonumber \\
 & \quad+2B_{\bar{x}\bar{x}}\left(\left(\frac{\partial^{2}\bar{x}}{\partial x\partial y}\frac{\partial\bar{x}}{\partial x}-\frac{\partial^{2}\bar{x}}{\partial x^{2}}\frac{\partial\bar{x}}{\partial y}\right)\dot{x}+\left(\frac{\partial^{2}\bar{x}}{\partial y^{2}}\frac{\partial\bar{x}}{\partial x}-\frac{\partial^{2}\bar{x}}{\partial x\partial y}\frac{\partial\bar{x}}{\partial y}\right)\dot{y}\right)\nonumber \\
 & \quad+2B_{\bar{x}\bar{y}}\left(\left(\frac{\partial^{2}\bar{x}}{\partial x\partial y}\frac{\partial\bar{y}}{\partial x}-\frac{\partial^{2}\bar{x}}{\partial x^{2}}\frac{\partial\bar{y}}{\partial y}+\frac{\partial^{2}\bar{y}}{\partial x\partial y}\frac{\partial\bar{x}}{\partial x}-\frac{\partial^{2}\bar{y}}{\partial x^{2}}\frac{\partial\bar{x}}{\partial y}\right)\dot{x}\right.\label{eq:A-transform}\\
 & \quad\quad\left.+\left(\frac{\partial^{2}\bar{x}}{\partial y^{2}}\frac{\partial\bar{y}}{\partial x}-\frac{\partial^{2}\bar{x}}{\partial x\partial y}\frac{\partial\bar{y}}{\partial y}+\frac{\partial^{2}\bar{y}}{\partial y^{2}}\frac{\partial\bar{x}}{\partial x}-\frac{\partial^{2}\bar{y}}{\partial x\partial y}\frac{\partial\bar{x}}{\partial y}\right)\dot{y}\right)\nonumber \\
 & \quad+2B_{\bar{y}\bar{y}}\left(\left(\frac{\partial^{2}\bar{y}}{\partial x\partial y}\frac{\partial\bar{y}}{\partial x}-\frac{\partial^{2}\bar{y}}{\partial x^{2}}\frac{\partial\bar{y}}{\partial y}\right)\dot{x}+\left(\frac{\partial^{2}\bar{y}}{\partial y^{2}}\frac{\partial\bar{y}}{\partial x}-\frac{\partial^{2}\bar{y}}{\partial x\partial y}\frac{\partial\bar{y}}{\partial y}\right)\dot{y}\right),\nonumber 
\end{align}
and
\begin{align}
 & B_{xx}=B_{\bar{x}\bar{x}}\left(\frac{\partial\bar{x}}{\partial x}\right)^{2}+2B_{\bar{x}\bar{y}}\frac{\partial\bar{x}}{\partial x}\frac{\partial\bar{y}}{\partial x}+B_{\bar{y}\bar{y}}\left(\frac{\partial\bar{y}}{\partial x}\right)^{2},\nonumber \\
 & B_{xy}=B_{\bar{x}\bar{x}}\frac{\partial\bar{x}}{\partial x}\frac{\partial\bar{x}}{\partial y}+B_{\bar{x}\bar{y}}\left(\frac{\partial\bar{x}}{\partial x}\frac{\partial\bar{y}}{\partial y}+\frac{\partial\bar{x}}{\partial y}\frac{\partial\bar{y}}{\partial x}\right)+B_{\bar{y}\bar{y}}\frac{\partial\bar{y}}{\partial x}\frac{\partial\bar{y}}{\partial y},\label{eq:B-transform}\\
 & B_{yy}=B_{\bar{x}\bar{x}}\left(\frac{\partial\bar{x}}{\partial y}\right)^{2}+2B_{\bar{x}\bar{y}}\frac{\partial\bar{x}}{\partial y}\frac{\partial\bar{y}}{\partial y}+B_{\bar{y}\bar{y}}\left(\frac{\partial\bar{y}}{\partial y}\right)^{2}.\nonumber 
\end{align}
With respect to the given chart transformation,
\begin{align*}
 & \frac{1}{2}\left(\frac{\partial A_{x}}{\partial\dot{y}}-\frac{\partial A_{y}}{\partial\dot{x}}\right)dx\wedge dy+\left(B_{xx}dx+B_{xy}dy\right)\wedge d\dot{x}+\left(B_{xy}dx+B_{yy}dy\right)\wedge d\dot{y}\\
 & =\frac{1}{2}\left(\frac{\partial A_{\bar{x}}}{\partial\dot{\bar{y}}}-\frac{\partial A_{\bar{y}}}{\partial\dot{\bar{x}}}\right)d\bar{x}\wedge d\bar{y}+\left(B_{\bar{x}\bar{x}}d\bar{x}+B_{\bar{x}\bar{y}}d\bar{y}\right)\wedge d\dot{\bar{x}}+\left(B_{\bar{x}\bar{y}}d\bar{x}+B_{\bar{y}\bar{y}}d\bar{y}\right)\wedge d\dot{\bar{y}},
\end{align*}
proving that $\alpha'$ given by local formula \eqref{eq:AlfaPrime}
is a~$2$-form on $\mathbb{R}\times T^{1}M$. By similar arguments,
the same is true for $\alpha_{0}$.
\end{proof}
\begin{lem}
\label{lem:Eta0}The equation $\alpha_{0}\wedge dt=d\mu_{0}$ has
always a solution $\mu_{0}$ defined on $\mathbb{R}\times T^{1}M$.
In a fibered chart $\left(V,\psi\right)$, $\psi=\left(t,x,y\right)$,
on $\mathbb{R}\times M$, $\mu_{0}$ is expressed by
\begin{align}
\mu_{0} & =-\left(A_{x}-\frac{1}{2}\left(\frac{\partial A_{x}}{\partial\dot{y}}-\frac{\partial A_{y}}{\partial\dot{x}}\right)\dot{y}\right)tdx-\left(A_{y}-\frac{1}{2}\left(\frac{\partial A_{y}}{\partial\dot{x}}-\frac{\partial A_{x}}{\partial\dot{y}}\right)\dot{x}\right)tdy\nonumber \\
 & -\left(B_{xx}\dot{x}+B_{xy}\dot{y}\right)td\dot{x}-\left(B_{xy}\dot{x}+B_{yy}\dot{y}\right)td\dot{y},\label{eq:Eta0}
\end{align}
and the horizontal component $h\mu_{0}$ of $\mu_{0}$, defined on
$\mathbb{R}\times T^{2}M$, is expressed by formula
\begin{align}
h\mu_{0} & =-\left(\varepsilon_{x}\dot{x}+\varepsilon_{y}\dot{y}\right)tdt,\label{eq:HorEta0}
\end{align}
where $\varepsilon_{x}$, $\varepsilon_{y}$ are given by \eqref{eq:LinearEpsilon}.
\end{lem}

\begin{proof}
Using the Helmholtz conditions \eqref{eq:HelmAB} and formulas \eqref{eq:A-transform},
\eqref{eq:B-transform}, the chart transformation shows that $\mu_{0}$
given by \eqref{eq:Eta0} defines a $1$-form on $\mathbb{R}\times T^{1}M$.
To verify that $\mu_{0}$ satisfies $\alpha_{0}\wedge dt=d\mu_{0}$,
we proceed by a straightforward calculation in a~chart.
\end{proof}
Now we analyze the equation $\alpha'=d\mu'$, where $\alpha'$ is
given by formula \eqref{eq:AlfaPrime}.\textcolor{black}{{} To this
purpose we define canonical }\textit{\textcolor{black}{local}}\textcolor{black}{{}
sections as follows. Let $\left(V,\psi\right)$, $\psi=\left(t,x,y\right)$,
be a~fixed fibered chart on $\mathbb{R}\times M$, and $\left(V^{1},\psi^{1}\right)$,
$\psi^{1}=\left(t,x,y,\dot{x},\dot{y}\right)$, be the associated
chart on $\mathbb{R}\times T^{1}M$. Put
\[
\pi_{1}^{1}\left(t,x,y,\dot{x},\dot{y}\right)=\left(t,x,y,\dot{y}\right),\quad s_{1,\nu}^{1}\left(t,x,y,\dot{y}\right)=\left(t,x,y,\nu,\dot{y}\right),
\]
and
\[
\pi_{2}^{1}\left(t,x,y,\dot{y}\right)=\left(t,x,y\right),\quad s_{2,\sigma}^{1}\left(t,x,y\right)=\left(t,x,y,\sigma\right).
\]
$\pi_{1}^{1}$ maps the chart domain $V^{1}$ onto its open subset
$V_{1}^{1}$, given by the equation $\dot{x}=0$, whereas $\pi_{2}^{1}$
maps $V_{1}^{1}$ onto the chart domain $V$ in $\mathbb{R}\times M$. }

\textcolor{black}{Let $K_{1}$ be a~local homotopy operator, acting
on (local) differential forms on $V^{1}\subset\mathbb{R}\times T^{1}M$
as
\begin{equation}
K_{1}\rho=\intop_{0}^{\dot{x}}\left(\pi_{1}^{1}\right)^{*}\left(s_{1,\nu}^{1}\right)^{*}\left(i_{\frac{\partial}{\partial\dot{x}}}\rho\right)d\nu,\label{eq:K1}
\end{equation}
and $K_{2}$ acts on forms on $V_{1}^{1}\subset V^{1}$ as
\begin{equation}
K_{2}\varrho=\intop_{0}^{\dot{y}}\left(\pi_{2}^{1}\right)^{*}\left(s_{2,\sigma}^{1}\right)^{*}\left(i_{\frac{\partial}{\partial\dot{y}}}\varrho\right)d\sigma,\label{eq:K2}
\end{equation}
where in formulas \eqref{eq:K1}, \eqref{eq:K2}, the integration
operation is applied on coefficients of the coressponding differential
form.}
\begin{thm}
\label{thm:K}Let $\alpha_{\varepsilon}$ be the Lepage equivalent
of $\varepsilon$, and $\alpha'$ is $2$-form on $T^{1}M$ given
by means of \eqref{eq:AlfaPrime}. If\textcolor{black}{{} $\left(V,\psi\right)$,
$\psi=\left(t,x,y\right)$, is a~fibered chart on $\mathbb{R}\times M$,
then}
\begin{equation}
\alpha'-\omega=d\kappa,\label{eq:Keq}
\end{equation}
where
\begin{equation}
\omega=\frac{1}{2}\left(\frac{\partial A_{x}}{\partial\dot{y}}-\frac{\partial A_{y}}{\partial\dot{x}}\right)_{\left(x,y,0,0\right)}dx\wedge dy,\label{eq:omega}
\end{equation}
and
\begin{align}
\kappa & =K_{1}\alpha'+\left(\pi_{1}^{1}\right)^{*}K_{2}\left(\left(s_{1,0}^{1}\right)^{*}\alpha'\right)\label{eq:kappa}\\
 & =- \left( \intop_{0}^{\dot{x}}B_{xx}\left(x,y,\nu,\dot{y}\right)d\nu + \intop_{0}^{\dot{y}}B_{xy}\left(x,y,0,\sigma\right)d\sigma \right) dx \nonumber \\
 &- \left( \intop_{0}^{\dot{x}}B_{xy}\left(x,y,\nu,\dot{y}\right)d\nu + \intop_{0}^{\dot{y}}B_{yy}\left(x,y,0,\sigma\right)d\sigma \right) dy.\nonumber 
\end{align}
\end{thm}

\begin{proof}
First, we prove the identity
\begin{equation}
\alpha'-\left(\pi_{1}^{1}\right)^{*}\left(s_{1,0}^{1}\right)^{*}\alpha'=d\left(K_{1}\alpha'\right).\label{eq:K1eq}
\end{equation}
Using \eqref{eq:AlfaPrime}, the left-hand side of \eqref{eq:K1eq}
is expressed as
\begin{align*}
 & \alpha'-\left(\pi_{1}^{1}\right)^{*}\left(s_{1,0}^{1}\right)^{*}\alpha'\\
 & =\frac{1}{2}\left(\frac{\partial A_{x}}{\partial\dot{y}}-\frac{\partial A_{y}}{\partial\dot{x}}\right)dx\wedge dy+\left(B_{xx}dx+B_{xy}dy\right)\wedge d\dot{x}+\left(B_{xy}dx+B_{yy}dy\right)\wedge d\dot{y}\\
 & -\frac{1}{2}\left(\frac{\partial A_{x}}{\partial\dot{y}}-\frac{\partial A_{y}}{\partial\dot{x}}\right)_{\left(x,y,0,\dot{y}\right)}dx\wedge dy-\left(B_{xy}\left(x,y,0,\dot{y}\right)dx+B_{yy}\left(x,y,0,\dot{y}\right)dy\right)\wedge d\dot{y}.
\end{align*}
From the definition of $K_{1}$ \eqref{eq:K1}, the right-hand side
of \eqref{eq:K1eq} reads
\begin{align*}
 & d\left(K_{1}\alpha'\right)=d\left(\intop_{0}^{\dot{x}}\left(\pi_{1}^{1}\right)^{*}\left(s_{1,\nu}^{1}\right)^{*}\left(i_{\frac{\partial}{\partial\dot{x}}}\alpha'\right)d\nu\right)\\
 & =\left(\intop_{0}^{\dot{x}}\left(\frac{\partial B_{xx}}{\partial y}-\frac{\partial B_{xy}}{\partial x}\right)_{\left(x,y,\nu,\dot{y}\right)}d\nu\right)dx\wedge dy+\left(B_{xx}dx+B_{xy}dy\right)\wedge d\dot{x}\\
 & -\left(\intop_{0}^{\dot{x}}\frac{\partial B_{xx}}{\partial\dot{y}}_{\left(x,y,\nu,\dot{y}\right)}d\nu\right)d\dot{y}\wedge dx-\left(\intop_{0}^{\dot{x}}\frac{\partial B_{xy}}{\partial\dot{y}}_{\left(x,y,\nu,\dot{y}\right)}d\nu\right)d\dot{y}\wedge dy.
\end{align*}
Now, we apply the Helmholtz conditions \eqref{eq:HelmAB} to $d\left(K_{1}\alpha'\right)$.
Namely, the identities \eqref{eq:Helmholtz-1}, and
\[
\frac{\partial B_{xx}}{\partial y}-\frac{\partial B_{xy}}{\partial x}=\frac{1}{2}\frac{\partial}{\partial\dot{x}}\left(\frac{\partial A_{x}}{\partial\dot{y}}-\frac{\partial A_{y}}{\partial\dot{x}}\right),
\]
imply that
\begin{align*}
 & d\left(K_{1}\alpha'\right)\\
 & =\frac{1}{2}\left(\intop_{0}^{\dot{x}}\frac{\partial}{\partial\dot{x}}\left(\frac{\partial A_{x}}{\partial\dot{y}}-\frac{\partial A_{y}}{\partial\dot{x}}\right)_{\left(x,y,\nu,\dot{y}\right)}d\nu\right)dx\wedge dy+\left(B_{xx}dx+B_{xy}dy\right)\wedge d\dot{x}\\
 & -\left(\intop_{0}^{\dot{x}}\frac{\partial B_{xy}}{\partial\dot{x}}_{\left(x,y,\nu,\dot{y}\right)}d\nu\right)d\dot{y}\wedge dx-\left(\intop_{0}^{\dot{x}}\frac{\partial B_{yy}}{\partial\dot{x}}_{\left(x,y,\nu,\dot{y}\right)}d\nu\right)d\dot{y}\wedge dy\\
 & =\frac{1}{2}\left(\frac{\partial A_{x}}{\partial\dot{y}}-\frac{\partial A_{y}}{\partial\dot{x}}\right)dx\wedge dy-\frac{1}{2}\left(\frac{\partial A_{x}}{\partial\dot{y}}-\frac{\partial A_{y}}{\partial\dot{x}}\right)_{\left(x,y,0,\dot{y}\right)}dx\wedge dy\\
 & +\left(B_{xx}dx+B_{xy}dy\right)\wedge d\dot{x}+\left(B_{xy}dx+B_{yy}dy\right)\wedge d\dot{y}\\
 & -\left(B_{xy}\left(x,y,0,\dot{y}\right)dx+B_{yy}\left(x,y,0,\dot{y}\right)dy\right)\wedge d\dot{y},
\end{align*}
as required to show \eqref{eq:K1eq}.

By similar arguments we observe that the following formula holds
\[
\left(s_{1,0}^{1}\right)^{*}\alpha'=\left(\pi_{2}^{1}\right)^{*}\left(s_{2,0}^{1}\right)^{*}\left(s_{1,0}^{1}\right)^{*}\alpha'+d\left(K_{2}\left(\left(s_{1,0}^{1}\right)^{*}\alpha'\right)\right).
\]
Hence
\begin{align*}
\left(\pi_{1}^{1}\right)^{*}\left(s_{1,0}^{1}\right)^{*}\alpha' & =\left(s_{1,0}^{1}\circ s_{2,0}^{1}\circ\pi_{2}^{1}\circ\pi_{1}^{1}\right)^{*}\alpha'+d\left(\left(\pi_{1}^{1}\right)^{*}K_{2}\left(\left(s_{1,0}^{1}\right)^{*}\alpha'\right)\right)\\
 & =\omega+d\left(\left(\pi_{1}^{1}\right)^{*}K_{2}\left(\left(s_{1,0}^{1}\right)^{*}\alpha'\right)\right)
\end{align*}
and substituting this formula into \eqref{eq:K1eq}, we get the identity
\eqref{eq:Keq}.
\end{proof}
\begin{rem}
The identity \eqref{eq:Keq} is formulated by Theorem \ref{thm:K}
in an arbitrary chart. By means of chart transformations, we show
that this formula holds also \emph{globally}. However, we emphasize
that even if $\omega$ \eqref{eq:omega} defines a~differential form
on $T^{1}M$, this need \emph{not} be longer true in general for a~solution
$\kappa$ of the equation \eqref{eq:Keq}. The well-known example
of a~differential form with an analogous property is the \emph{Cartan
form} $\Theta_{\lambda}$, which depends on the choice of a~Lagrangian
$\lambda$ whereas $d\Theta_{\lambda}$ does not.
\end{rem}

\begin{thm}
\label{thm:KappaGlobal}Both $\kappa$ \eqref{eq:kappa} and $\omega$
\eqref{eq:omega} define (global) differential $1$-forms on $T^{1}M$.
\end{thm}

\begin{proof}
As for the form $\omega$, the transformation property \eqref{eq:A-transform}
shows that \eqref{eq:omega} defines a~$2$-form on $T^{1}M\subset\mathbb{R}\times T^{1}M$.
We now prove that also $\kappa$, given in an arbitrary fibered chart
by means of the formula \eqref{eq:kappa}, is a~global form. To this
purpose, consider two overlapping fibered charts on $T^{1}M$ with
coordinates functions $\left(x,y,\dot{x},\dot{y}\right)$ and $\left(\bar{x},\bar{y},\dot{\bar{x}},\dot{\bar{y}}\right)$,
and the coordinate transformation $\Psi\circ\bar{\Psi}^{-1}\left(\bar{x},\bar{y},\dot{\bar{x}},\dot{\bar{y}}\right)=\left(x\left(\bar{x},\bar{y}\right),y\left(\bar{x},\bar{y}\right),\dot{x}\left(\bar{x},\bar{y},\dot{\bar{x}},\dot{\bar{y}}\right),\dot{y}\left(\bar{x},\bar{y},\dot{\bar{x}},\dot{\bar{y}}\right)\right)$.
Putting
\begin{align*}
 & f\left(x,y,\dot{x},\dot{y}\right)=-\intop_{0}^{\dot{x}}B_{xx}\left(x,y,\nu,\dot{y}\right)d\nu-\intop_{0}^{\dot{y}}B_{xy}\left(x,y,0,\sigma\right)d\sigma,\\
 & g\left(x,y,\dot{x},\dot{y}\right)=-\intop_{0}^{\dot{x}}B_{xy}\left(x,y,\nu,\dot{y}\right)d\nu-\intop_{0}^{\dot{y}}B_{yy}\left(x,y,0,\sigma\right)d\sigma,
\end{align*}
we get an expression of $\kappa$ by means of the coordinates $\left(x,y,\dot{x},\dot{y}\right)$
as
\begin{equation}
\kappa=f\left(x,y,\dot{x},\dot{y}\right)dx+g\left(x,y,\dot{x},\dot{y}\right)dy.\label{eq:Kappa-puvodni}
\end{equation}
{\small{}}%
Using the change of variables theorem for integrals, formulas \eqref{eq:B-transform},
and the transformation described by
\begin{align*}
 & \left(\bar{\Psi}\circ\Psi^{-1}\right)\left(x,y,\nu,\dot{y}\right)=\left(\bar{x},\bar{y},\bar{\nu},\bar{\mu}\right),\quad\bar{x}=\bar{x}\left(x,y\right),\quad\bar{y}=\bar{y}\left(x,y\right),\\
 & \bar{\nu}=\frac{\partial\bar{x}}{\partial x}\nu+\frac{\partial\bar{x}}{\partial y}\dot{y},\quad\bar{\mu}=\frac{\partial\bar{y}}{\partial x}\nu+\frac{\partial\bar{y}}{\partial y}\dot{y},\quad\nu=\frac{\partial x}{\partial\bar{x}}\bar{\nu}+\frac{\partial x}{\partial\bar{y}}\bar{\mu},\quad\dot{y}=\frac{\partial y}{\partial\bar{x}}\bar{\nu}+\frac{\partial y}{\partial\bar{y}}\bar{\mu},\\
 & \left(\bar{\Psi}\circ\Psi^{-1}\right)\left(x,y,0,\sigma\right)=\left(\bar{x},\bar{y},\bar{\tau},\bar{\sigma}\right),\\
 & \bar{\tau}=\frac{\partial\bar{x}}{\partial y}\sigma,\quad\bar{\sigma}=\frac{\partial\bar{y}}{\partial y}\sigma,\quad0=\frac{\partial x}{\partial\bar{x}}\bar{\tau}+\frac{\partial x}{\partial\bar{y}}\bar{\sigma},\quad\sigma=\frac{\partial y}{\partial\bar{x}}\bar{\tau}+\frac{\partial y}{\partial\bar{y}}\bar{\sigma},
\end{align*}
and
\[
\frac{\partial y}{\partial\bar{x}}d\bar{\nu}+\frac{\partial y}{\partial\bar{y}}d\bar{\mu}=0,\quad\frac{\partial x}{\partial\bar{x}}d\bar{\tau}+\frac{\partial x}{\partial\bar{y}}d\bar{\sigma}=0,
\]
the integral \eqref{eq:Kappa-puvodni} over segments transforms into
the \emph{line} integral
\begin{align*}
 & \left(\Psi\circ\bar{\Psi}^{-1}\right)^{*}\left(f\left(x,y,\dot{x},\dot{y}\right)dx+g\left(x,y,\dot{x},\dot{y}\right)dy\right)\\
 & =-\intop_{\bar{\nu}=0,\bar{\mu}=0}^{\bar{\nu}=\dot{\bar{x}},\bar{\mu}=\dot{\bar{y}}}\left(B_{\bar{x}\bar{x}}\left(\bar{x},\bar{y},\bar{\nu},\bar{\mu}\right)d\bar{\nu}+B_{\bar{x}\bar{y}}\left(\bar{x},\bar{y},\bar{\nu},\bar{\mu}\right)d\bar{\mu}\right)\cdot d\bar{x}\\
 & -\intop_{\bar{\nu}=0,\bar{\mu}=0}^{\bar{\nu}=\dot{\bar{x}},\bar{\mu}=\dot{\bar{y}}}\left(B_{\bar{x}\bar{y}}\left(\bar{x},\bar{y},\bar{\nu},\bar{\mu}\right)d\bar{\nu}+B_{\bar{y}\bar{y}}\left(\bar{x},\bar{y},\bar{\nu},\bar{\mu}\right)d\bar{\mu}\right)\cdot d\bar{y},
\end{align*}
Now we observe that both summands in the previous expression of $\left(\Psi\circ\bar{\Psi}^{-1}\right)^{*}\kappa$
represent line integrals that are \emph{independent} upon the choice
of a~path as the integrands satisfy the \emph{Helmholtz conditions}
\eqref{eq:Helmholtz-1}, cf. Theorem \ref{thm:LocalVariationality}.
To this purpose we may consider, without loss of generality, rectangular
charts. Hence
\begin{align*}
 & \left(\Psi\circ\bar{\Psi}^{-1}\right)^{*}\left(f\left(x,y,\dot{x},\dot{y}\right)dx+g\left(x,y,\dot{x},\dot{y}\right)dy\right)\\
 & =-\intop_{0}^{\dot{\bar{x}}}B_{\bar{x}\bar{x}}\left(\bar{x},\bar{y},\bar{\nu},\dot{\bar{y}}\right)d\bar{\nu}\cdot d\bar{x}-\intop_{0}^{\dot{\bar{y}}}B_{\bar{x}\bar{y}}\left(\bar{x},\bar{y},0,\bar{\mu}\right)d\bar{\mu}\cdot d\bar{x}\\
 & -\intop_{0}^{\dot{\bar{x}}}B_{\bar{x}\bar{y}}\left(\bar{x},\bar{y},\bar{\nu},\dot{\bar{y}}\right)d\bar{\nu}\cdot d\bar{x}-\intop_{0}^{\dot{\bar{y}}}B_{\bar{y}\bar{y}}\left(\bar{x},\bar{y},0,\bar{\mu}\right)d\bar{\mu}\cdot d\bar{x}\\
 & =f\left(\bar{x},\bar{y},\dot{\bar{x}},\dot{\bar{y}}\right)dx+g\left(\bar{x},\bar{y},\dot{\bar{x}},\dot{\bar{y}}\right)dy,
\end{align*}
as required.
\end{proof}
\begin{rem}
Theorems \ref{thm:K} and \ref{thm:KappaGlobal} show, according to
our expectation from the cohomology results of the global variational
theory by Takens \cite{Takens} and others, that the global exactness
of Lepage equivalent $\alpha_{\varepsilon}$ on $J^{2}\left(\mathbb{R}\times M\right)$
reduces to global exactness of a~$2$-form \emph{defined on} $M$.
\end{rem}

\begin{cor}
\label{cor:Simple}If the $2$-form $\omega$ \eqref{eq:omega} vanishes,
i.e. if the coefficients of $\omega$ satisfy
\begin{equation}
\left(\frac{\partial A_{x}}{\partial\dot{y}}-\frac{\partial A_{y}}{\partial\dot{x}}\right)_{\left(x,y,0,0\right)}=0\label{eq:Simple}
\end{equation}
in every chart, then source form $\varepsilon$ is globally variational
and it admits a~Lagrangian on $\mathbb{R}\times T^{1}M$, namely
\begin{equation}
\lambda=h\left(\mu_{0}+\kappa\right),\label{eq:SimpleLagrangian}
\end{equation}
where $\mu_{0}$ and $\kappa$ are given by \eqref{eq:Eta0} and \eqref{eq:kappa},
respectively.
\end{cor}

\begin{proof}
This is an immediate consequence of Theorem \ref{thm:K} and Lemma
\ref{lem:Eta0}.
\end{proof}

\section{A global construction on 2-manifolds: Top-cohomology\label{sec:Top}}

Recall now two theorems, characterizing the \emph{top} \emph{de Rham
cohomology} groups of connected smooth manifolds, that is $H_{\mathrm{dR}}^{m}M$,
where $\dim M=m$. Our main reference is Lee \cite{Lee}.
\begin{thm}[Orientable, top-cohomology]
\label{thm:TopOrient} Let $M$ be a connected orientable smooth
$m$-manifold.

\emph{(a)} If $M$ is compact, then $H_{\mathrm{dR}}^{m}M$ is one-dimensional
and is spanned by the cohomology class of any smooth orientation form.

\emph{(b)} If $M$ is noncompact, then $H_{\mathrm{dR}}^{m}M=0$.
\end{thm}

\begin{thm}[Nonorientable, top-cohomology]
\label{thm:TopNonOrient} Let M be a connected nonorientable smooth
n-manifold. Then $H_{\mathrm{dR}}^{m}M=0$.
\end{thm}

\begin{rem}
In other words, Theorems \ref{thm:TopOrient} and \ref{thm:TopNonOrient}
characterize manifolds with trivial de Rham top-cohomology groups
as follows: \emph{If a~connected smooth $m$-manifold $M$ obeys
$H_{\mathrm{dR}}^{m}M=0$, then $M$ is either nonorientable, or orientable
and noncompact.} Moreover, the proof of Theorem \ref{thm:TopOrient},
(b), and also of Theorem \ref{thm:TopNonOrient}, is \emph{constructive.}
\end{rem}

\begin{rem}
\label{rem:Lee}In general, if $M$ is an $m$-dimensional smooth
manifold and $\rho$ is a~\emph{closed} differential $k$-form on
$M$, $k\leq m$, then the equation $\rho=d\mu$ need \emph{not} have
a~(global) solution $\mu$ on $X$. Indeed, it is the $k$-th de
Rham cohomology group $H_{\mathrm{dR}}^{k}M=\mathrm{Ker}\,d_{k}/\mathrm{Im}\,d_{k-1}$
which decides about \emph{solvability} of the exactness equation $\rho=d\mu$.
Clearly, if $H_{\mathrm{dR}}^{k}M$ is \emph{trivial}, then $\rho=d\mu$
has always a solution $\mu$ on $X$. Nevertheless, in this case ($H_{\mathrm{dR}}^{k}M=0$)
there is \emph{no general constructive} procedure of finding a solution
$\mu$ for a given closed $k$-form $\rho$, where $k<m$; if $k=m$,
to find a solution one can apply the \emph{top-cohomology} theorems.
\end{rem}

Our problem of solving the exactness equation $\alpha_{\varepsilon}=d\mu$
\eqref{eq:AlfaExactness} \emph{globally} concerns a~given $2$-form
$\alpha_{\varepsilon}$ on $\mathbb{R}\times T^{1}M$. However, Theorem
\ref{thm:K} and Lemma \ref{lem:Eta0} reduce this problem to a~$2$-form,
which is defined on $M$. Indeed, applying formulas \eqref{eq:Eta0}
and \eqref{eq:Keq} we obtain 
\[
\alpha_{\varepsilon}=\alpha_{0}+\alpha'=\omega+d\left(\mu_{0}+\kappa\right),
\]
where $\mu_{0}$ \eqref{eq:Eta0} and $\kappa$ \eqref{eq:kappa}
are $1$-forms on $\mathbb{R}\times T^{1}M$, and $\omega$ \eqref{eq:omega}
is a~$2$-form on $M$. 

Following Lee \cite{Lee}, for orientable noncompact manifolds we
briefly describe construction of a solution of the exactness equation
for $\omega$,
\begin{equation}
\omega=d\eta.\label{eq:OmegaExactness}
\end{equation}
To this purpose, recall the Poincar\'e lemma for \emph{compactly
supported} forms.
\begin{lem}
\label{lem:Poincare}Let $\rho$ be a~compactly supported closed
$k$-form on $\mathbb{R}^{m}$, where $1\leq k\leq n$. If $k=m$,
suppose in addition that $\int_{\mathbb{R}^{m}}\rho=0$. Then there
exists a~compactly supported $(k-1)$-form $\vartheta$ on $\mathbb{R}^{m}$
such that $d\vartheta=\rho$.
\end{lem}

Let us consider the case when $M$ is a \emph{connected orientable
noncompact} smooth $2$-manifold (Theorem \eqref{thm:TopOrient},
(b)). The solution $\eta$ of \eqref{eq:OmegaExactness} is determined
by means of an appropriate covering of $M$.
\begin{lem}
\label{lem:Covering}If $M$ is a noncompact connected manifold, then
there exists countable, locally finite open cover $\{V_{j}\}$ of
$M$ such that each $V_{j}$ is connected and precompact, and for
each $j$, there exists an index $k>j$ such that $V_{j}\cap V_{k}\neq\emptyset$.
\end{lem}

\begin{proof}
See Lee \cite{Lee}, Errata 2018.
\end{proof}
Consider $\{V_{j}\}$ satisfying Lemma \ref{lem:Covering}. For each
$j$, denote $K(j)$ the least index $k>j$ such that $V_{j}\cap V_{k}\neq\emptyset$.
Let $\theta_{j}$ be a~$2$-form, compactly supported in $V_{j}\cap V_{K(j)}$
such that $\int_{M}\theta_{j}=1$. Let $\{\psi_{j}\}$ be a~smooth
partition of unity subordinate to $\{V_{j}\}$. For each $j$, put
$\omega_{j}=\psi_{j}\omega$. Put $c_{1}=\int_{V_{1}}\omega_{1}$.
Then the $2$-form $\omega_{1}-c_{1}\theta_{1}$ is compactly supported
in $V_{1}$, and $\int_{M}\left(\omega_{1}-c_{1}\theta_{1}\right)=0$.
Withous loss of generality, we may assume that $V_{j}$ is star-shaped
for every $j$. By Lemma \ref{lem:Poincare} there exists a compactly
supported~$1$-form $\eta_{1}$ on $V_{1}$ such that $d\eta_{1}=\omega_{1}-c_{1}\theta_{1}$.
Further, we proceed by induction as follows. Suppose we have compactly
supported $1$-forms $\eta_{j}$ on $V_{j}$ and constants $c_{j}$,
where $1\leq j\leq m$, such that
\begin{equation}
d\eta_{j}=\left(\omega_{j}+\sum_{i:K(i)=j}c_{i}\theta_{i}\right)-c_{j}\theta_{j}.\label{eq:Lee}
\end{equation}
Let
\[
c_{j+1}=\int_{V_{j+1}}\left(\omega_{j+1}+\sum_{i:K(i)=j+1}c_{i}\theta_{i}\right).
\]
Then
\[
\int_{V_{j+1}}\left(\left(\omega_{j+1}+\sum_{i:K(i)=j+1}c_{i}\theta_{i}\right)-c_{j+1}\theta_{j+1}\right)=0,
\]
and Lemma \ref{lem:Poincare} assures existence of a~compactly supported
$1$-form $\eta_{j+1}$ on $V_{j+1}$ which satisfies formula \eqref{eq:Lee}
with $j$ replaced by $j+1$.

Now, extending each $\eta_{j}$ on $M\backslash V_{j}$ to be zero,
we set
\begin{equation}
\eta=\sum_{j}\eta_{j}.\label{eq:Eta-Reseni}
\end{equation}
Since the open covering $\{V_{j}\}$ is locally finite, formula \eqref{eq:Eta-Reseni}
is correct and defines a~differential $1$-form on $M$ such that
$d\eta=\omega$, as required.

The preceding procedure of solving the exactness equation \eqref{eq:OmegaExactness},
together with Lemma \ref{lem:AlfaClosed} and Theorems \ref{thm:K},
\ref{thm:KappaGlobal}, imply our main result, the following application of the de
Rham top-cohomology in the theory of variational equations.
\begin{thm}
\label{thm:Top-Var}Let $M$ be a connected smooth $2$-manifold,
and let $\varepsilon$ be a locally variational source form defined
on $\mathbb{R}\times T^{2}M$. If $H_{\mathrm{dR}}^{2}M=0$, then
$\varepsilon$ is globally variational and admits a global Lagragian
\[
\lambda=h\left(\mu_{0}+\kappa+\eta\right),
\]
where~$1$-form $\eta$ is a~solution of equation \eqref{eq:OmegaExactness}
on $M$, $\mu_{0}$ and $\kappa$ are given by formulas \eqref{eq:Eta0}
and \eqref{eq:kappa}, respectively.
\end{thm}

\section{Examples}

We analyze two examples of globally variational source forms on $\mathbb{R}\times T^{2}M$
over $\mathbb{R}$ (one-dimensional basis), where $M$ is \emph{two}-dimensional
smooth manifold, the open M\"obius strip $M_{r,a}$ and the punctured
torus $T_{P}$, both with \emph{trivial} the second de Rham cohomology
group, $H_{\mathrm{dR}}^{2}M=0$. Another classical examples of smooth
manifolds with this property include namely the punctured plane $\mathbb{R}^{2}\backslash\{0\}$,
and the Klein bottle $K$.
\begin{example}[Kinetic energy on $M_{r,a}$]
Consider the open subset $W\subset\mathbb{R}^{4}$ of the Euclidean
space, where $W=\mathbb{R}\times\left(\mathbb{R}^{3}\backslash\{(0,0,z)\}\right)$,
endowed with its open submanifold structure and global Cartesian coordinates
$\left(t,x,y,z\right)$. We introduce an atlas on $W$, adapted to
fibered open M\"obius strip $\mathbb{R}\times M_{r,a}$ of radius
$r$ and width $2a$, $0<a<r$. Let $V$ and $\bar{V}$ be an open
covering of $W$, given by
\[
V=\mathbb{R}\times\left(\mathbb{R}^{3}\backslash\left((-\infty,0]\times\{0\}\times\mathbb{R}\right)\right),\quad\bar{V}=\mathbb{R}\times\left(\mathbb{R}^{3}\backslash\left([0,\infty)\times\{0\}\times\mathbb{R}\right)\right),
\]
and define coordinate functions $\left(t,\varphi,\tau,\vartheta\right)$
on $V$ by $t=t$,
\begin{align*}
\varphi & =\mathrm{atan}2(y,x),\\
\tau & =\frac{1}{\sqrt{2}}\left(\sqrt{x^{2}+y^{2}}-r\right)\sqrt{1+\frac{x}{\sqrt{x^{2}+y^{2}}}}+\frac{1}{\sqrt{2}}\textrm{sgn}(y)z\sqrt{1-\frac{x}{\sqrt{x^{2}+y^{2}}}},\\
\vartheta & =-\frac{1}{\sqrt{2}}\left(\sqrt{x^{2}+y^{2}}-r\right)\textrm{sgn}(y)\sqrt{1-\frac{x}{\sqrt{x^{2}+y^{2}}}}+\frac{1}{\sqrt{2}}z\sqrt{1+\frac{x}{\sqrt{x^{2}+y^{2}}}},
\end{align*}
and $\left(\bar{t},\bar{\varphi},\bar{\tau},\bar{\vartheta}\right)$
on $\bar{V}$ by $\bar{t}=t$, and
\[
\bar{\varphi}=\begin{cases}
\mathrm{atan}2(y,x), & y\geq 0,\\
\mathrm{atan}2(y,x)+2\pi, & y<0,
\end{cases}\ \ 
\bar{\tau}=\begin{cases}
\tau\,\mathrm{sgn}(y), & y\neq 0,\\
z, & y=0,
\end{cases}\ \ 
\bar{\vartheta}=\begin{cases}
\vartheta\,\mathrm{sgn}(y), & y\neq 0,\\
x+r, & y=0,
\end{cases}
\]
where $\mathrm{atan}2(y,x)$ is the arctangent function with two arguments.
One can directly check that the pairs $\left(V,\Psi\right)$, $\Psi=\left(t,\varphi,\tau,\vartheta\right)$,
and $\left(\bar{V},\bar{\Psi}\right)$, $\bar{\Psi}=\left(\bar{t},\bar{\varphi},\bar{\tau},\bar{\vartheta}\right)$,
are charts on $W$ \emph{adapted} to $\mathbb{R}\times M_{r,a}$,
constituting a~smooth atlas on $W$. In the chart $\left(V,\Psi\right)$
(resp. $\left(\bar{V},\bar{\Psi}\right)$), $\mathbb{R}\times M_{r,a}$
has the equation $\vartheta=0$ with $-a<\tau<a$ (resp. $\bar{\vartheta}=0$
with $-a<\bar{\tau}<a$ ). On the intersection $V\cap\bar{V}$, the
chart transformations between $\left(V,\Psi\right)$ and $\left(\bar{V},\bar{\Psi}\right)$
is given by
\[
\Psi\circ\bar{\Psi}^{-1}:\bar{\Psi}(\bar{V})\backslash\{\bar{\varphi}=\pi\}\rightarrow\Psi(V)\backslash\{\varphi=0\},
\]
\begin{equation}
\Psi\circ\bar{\Psi}^{-1}(\bar{t},\bar{\varphi},\bar{\tau},\bar{\vartheta})=\begin{cases}
(\bar{t},\bar{\varphi},\bar{\tau},\bar{\vartheta}), & \bar{\varphi}\in(0,\pi),\\
(\bar{t},\bar{\varphi}-2\pi,-\bar{\tau},-\bar{\vartheta}), & \bar{\varphi}\in(\pi,2\pi),
\end{cases}\label{eq:MobiusTransform1}
\end{equation}

Let $\varepsilon$ be a~source form on $\mathbb{R}\times T^{1}M_{r,a}$,
locally expressed by 
\[
\varepsilon=\varepsilon_{\varphi}\omega^{\varphi}\wedge dt+\varepsilon_{\tau}\omega^{\tau}\wedge dt,
\]
where $\varepsilon_{\varphi}=A_{\varphi}+B_{\varphi\varphi}\ddot{\varphi}$
, $\varepsilon_{\tau}=A_{\tau}+B_{\tau\tau}\ddot{\tau}$, and
\begin{align*}
B_{\varphi\varphi} & =-\left(\left(r+\tau\cos\frac{\varphi}{2}\right)^{2}+\frac{\tau^{2}}{4}\right),\quad B_{\tau\tau}=-1,\\
A_{\varphi} & =\frac{1}{2}\dot{\varphi}^{2}\tau\sin\frac{\varphi}{2}\left(r+\tau\cos\frac{\varphi}{2}\right)-\frac{1}{2}\dot{\varphi}\dot{\tau}\left(4\cos\frac{\varphi}{2}\left(r+\tau\cos\frac{\varphi}{2}\right)+\tau\right),\\
A_{\tau} & =\frac{1}{4}\dot{\varphi}^{2}\left(4\cos\frac{\varphi}{2}\left(r+\tau\cos\frac{\varphi}{2}\right)+\tau\right).
\end{align*}
Using the chart transformation \eqref{eq:MobiusTransform1}, it is
easy to verify that $\varepsilon$ defines a~2-form on $\mathbb{R}\times T^{2}M_{r,a}$.
Since
\[
\left(\frac{\partial A_{\varphi}}{\partial\dot{\tau}}-\frac{\partial A_{\tau}}{\partial\dot{\varphi}}\right)_{\left(\varphi,\tau,0,0\right)}=\left(-\dot{\varphi}\left(4\cos\frac{\varphi}{2}\left(r+\tau\cos\frac{\varphi}{2}\right)+\tau\right)\right)_{\left(\varphi,\tau,0,0\right)}=0,
\]
Corollary \ref{cor:Simple} implies that $\varepsilon$ admits a global
Lagrangian
\[
\lambda=h\left(\mu_{0}+\kappa\right),
\]
where $\mu_{0}$ \eqref{eq:Eta0} reads
\begin{align*}
\mu_{0} & =-\left(A_{\varphi}+\frac{1}{2}\dot{\varphi}\dot{\tau}\left(4\cos\frac{\varphi}{2}\left(r+\tau\cos\frac{\varphi}{2}\right)+\tau\right)\right)td\varphi\\
 & -\left(A_{\tau}-\frac{1}{2}\dot{\varphi}^{2}\left(4\cos\frac{\varphi}{2}\left(r+\tau\cos\frac{\varphi}{2}\right)+\tau\right)\right)td\tau\\
 & -B_{\varphi\varphi}\dot{\varphi}td\dot{\varphi}-B_{\tau\tau}\dot{\tau}td\dot{\tau},
\end{align*}
and $\kappa$ \eqref{eq:kappa} is given by
\[
\kappa=K_{1}\alpha'+\left(\pi_{1}^{1}\right)^{*}K_{2}\left(\left(s_{1,0}^{1}\right)^{*}\alpha'\right),
\]
where from \eqref{eq:AlfaPrime},
\begin{align*}
\alpha' & =-\frac{1}{2}\dot{\varphi}\left(4\cos\frac{\varphi}{2}\left(r+\tau\cos\frac{\varphi}{2}\right)+\tau\right)d\varphi\wedge d\tau+B_{\varphi\varphi}d\varphi\wedge d\dot{\varphi}+B_{\tau\tau}d\tau\wedge d\dot{\tau}.
\end{align*}
From \eqref{eq:K1}, \eqref{eq:K2}, we get 
\[
\kappa=\left(\left(r+\tau\cos\frac{\varphi}{2}\right)^{2}+\frac{\tau^{2}}{4}\right)\dot{\varphi}d\varphi+\dot{\tau}d\tau.
\]
Hence a~global Lagrangian on $\mathbb{R}\times T^{2}M_{r,a}$ reads
\begin{align*}
\lambda & =\mathscr{L}dt,
\end{align*}
where $\mathscr{L}:\mathbb{R}\times T^{2}M_{r,a}\rightarrow\mathbb{R}$
is the Lagrange function, given by
\begin{align*}
\mathscr{L} & =-\frac{1}{2}\dot{\varphi}^{2}\tau\sin\frac{\varphi}{2}\left(r+\tau\cos\frac{\varphi}{2}\right)\dot{\varphi}t+\frac{1}{4}\dot{\varphi}^{2}\left(4\cos\frac{\varphi}{2}\left(r+\tau\cos\frac{\varphi}{2}\right)+\tau\right)\dot{\tau}t\\
 & +\left(\left(r+\tau\cos\frac{\varphi}{2}\right)^{2}+\frac{\tau^{2}}{4}\right)\dot{\varphi}\ddot{\varphi}t+\dot{\tau}\ddot{\tau}t+\left(\left(r+\tau\cos\frac{\varphi}{2}\right)^{2}+\frac{\tau^{2}}{4}\right)\dot{\varphi}^{2}+\dot{\tau}^{2}.
\end{align*}
Note that $\varepsilon$ admits also the kinetic energy Lagrangian
$\mathscr{L}_{kin}dt$, given by a~global function on $\mathbb{R}\times T^{1}M_{r,a}$,
\[
\mathscr{L}_{kin}=\frac{1}{2}\left(\dot{\tau}^{2}+\left(\left(r+\tau\cos\frac{\varphi}{2}\right)^{2}+\frac{\tau^{2}}{4}\right)\dot{\varphi}^{2}\right),
\]
which arises as a~\emph{pull-back} of the standard kinetic Lagrangian
on $J^{1}\left(\mathbb{R}\times\mathbb{R}^{3}\right)$ with respect
to the canonical embedding of $M_{r,a}$ into the Euclidean space
$\mathbb{R}^{3}$, and which is\emph{ equivalent} to $\lambda$ (i.e.
the associated Euler-Lagrange forms coincide).
\end{example}

\begin{example}[Gyroscopic equations on punctured torus]
 In this example, we study \emph{gyroscopic type} equations on the
\emph{punctured torus} $T_{P}$. The torus $T=S_{R}^{1}\times S_{r}^{1}$
(Cartesian products of circles of radius $R$ and $r$, respectively)
in the Euclidean space $\mathbb{R}^{3}$ is endowed with its smooth
manifold structure as the Cartesian product of smooth structures on
$S^{1}$, and the punctured torus $T_{P}=T\setminus\{(R+r,0,0)\}$
has the open submanifold structure. The parametric equations of $T\subset\mathbb{R}^{3}$
reads
\begin{align*}
 & x=\left(R+r\cos\vartheta\right)\cos\varphi,\quad y=\left(R+r\cos\vartheta\right)\sin\varphi,\quad z=r\sin\vartheta,
\end{align*}
where $0<r<R$, $\varphi\in[0,2\pi)$, and $\vartheta\in[0,2\pi)$.
The point $P$ with Euclidean coordinates $(R+r,0,0)$ arises for
$\varphi=0=\vartheta$. A~smooth atlas on $T_{P}$ can be chosen
as the following four charts, $\left(U_{\pi\pi},\Phi_{\pi\pi}\right)$,
$\Phi_{\pi\pi}=\left(\varphi,\vartheta\right)$, where $\varphi$
and $\vartheta$ are the angle coordinates on $S_{R}^{1}$ and $S_{r}^{1}$,
respectively, and $\left(U_{00},\Phi_{00}\right)$, $\Phi_{00}=\left(\bar{\varphi},\bar{\vartheta}\right)$,
$\left(U_{\pi0},\Phi_{\pi0}\right)$, $\Phi_{\pi0}=\left(\varphi,\bar{\vartheta}\right)$,
$\left(U_{0\pi},\Phi_{0\pi}\right)$, $\Phi_{0\pi}=\left(\bar{\varphi},\vartheta\right)$,
where $-\pi<\varphi,\vartheta<\pi$, and $0<\bar{\varphi},\bar{\vartheta}<2\pi$.

Consider the system
\begin{equation}
\varepsilon_{\varphi}=A_{\varphi}+B_{\varphi\varphi}\ddot{\varphi}+B_{\varphi\vartheta}\ddot{\vartheta},\quad\varepsilon_{\vartheta}=A_{\vartheta}+B_{\varphi\vartheta}\ddot{\varphi}+B_{\vartheta\vartheta}\ddot{\vartheta},\label{eq:GyroscopicSystem}
\end{equation}
where $B_{\varphi\varphi}=\left(R+r\cos\vartheta\right)^{2}$, $B_{\vartheta\vartheta}=r^{2}$, $B_{\varphi\vartheta}=0$, and
\begin{align*}
 & A_{\varphi}=-r\left(R+r\cos\vartheta\right)\left(2\dot{\varphi}\sin\vartheta+a\sin\vartheta-b\sin\varphi\cos\vartheta+c\cos\varphi\cos\vartheta\right)\dot{\vartheta},\\
 & A_{\vartheta}=r\left(R+r\cos\vartheta\right)\left(\dot{\varphi}\sin\vartheta+a\sin\vartheta-b\sin\varphi\cos\vartheta+c\cos\varphi\cos\vartheta\right)\dot{\varphi},
\end{align*}
and $a$, $b$, $c$ are some functions depending on $\varphi,\vartheta$.
System \eqref{eq:GyroscopicSystem} defines a~differential form $\varepsilon=(\varepsilon_{\varphi}\omega^{\varphi}+\varepsilon_{\vartheta}\omega^{\vartheta})\wedge dt$
on $\mathbb{R}\times T^{2}T_{P}$ \emph{globally}, and arises as the
pull-back of the gyroscopic type system in the Euclidean space $\mathbb{R}\times\mathbb{R}^{3}$,
\[
\ddot{x}=a\dot{y}+b\dot{z},\quad\ddot{y}=-a\dot{x}+c\dot{z},\quad\ddot{z}=-b\dot{x}-c\dot{y}.
\]
It is straightforward to verify with the help of Theorem \ref{thm:LocalVariationality},
\eqref{eq:HelmAB}, that \eqref{eq:GyroscopicSystem} is locally variational (cf. Krupka, Urban, and Voln\'a \cite{KUV}).
Theorem \ref{thm:Top-Var} implies that \eqref{eq:GyroscopicSystem}
is also globally variational and admits a~global Lagrangian of the
form
\begin{equation}\label{Last}
\lambda=h\left(\mu_{0}+\kappa+\eta\right),
\end{equation}
where $\mu_{0}$ and $\kappa$ are given by \eqref{eq:Eta0} and \eqref{eq:kappa},
and $\eta$ is a~solution of the equation $\omega=d\eta$ on $T_{P}$,
where $2$-form $\omega$ is given by \eqref{eq:omega}. Thus, we
have
\begin{align*}
 & \mu_{0}=r\left(R+r\cos\vartheta\right)\dot{\varphi}^2\sin\vartheta td\vartheta-\left(R+r\cos\vartheta\right)^{2}\dot{\varphi}td\dot{\varphi}-r^{2}\dot{\vartheta}td\dot{\vartheta},\\
 & \kappa=-\left(R+r\cos\vartheta\right)^{2}\dot{\varphi}d\varphi-r^{2}\dot{\vartheta}d\vartheta,
\end{align*}
and
\[
\omega=-r\left(R+r\cos\vartheta\right)\left(a\sin\vartheta-b\sin\varphi\cos\vartheta+c\cos\varphi\cos\vartheta\right)d\varphi\wedge d\vartheta.
\]
Clearly, for a~particular choice of the functions $a$, $b$, $c$,
for instance $a=\cos\vartheta\sin\varphi$, $b=\sin\vartheta+\cos\varphi$,
$c=\sin\varphi$, the $2$-form $\omega$ vanishes, hence by Corollary
\ref{cor:Simple} we get a~global Lagrangian $\lambda=h\left(\mu_{0}+\kappa\right)$.

In case that the $2$-form $\omega$ on $T_{P}$ does \emph{not} vanish,
we may proceed as described in Section 4 to construct a~solution
$\eta$ of $\omega=d\eta$, defined on $T_{P}$. To this purpose,
for instance, one may consider an open \emph{star-shaped} cover $\{V_{j}\}$
of $T_{P}$ satisfying conditions of Lemma \ref{lem:Covering} such that in addition each $V_j$ is a~preimage of rectangles with respect to coordinate mappings and obeys the condition $V_j\cap V_{j+1}\neq\emptyset$; e.g. $\{V_{j}\}$ arises as a~refinement of an open cover $\{\tilde{V}_{k}\}$ of $T_{P}$,
where $\tilde{V}_{0}=T_{P}\setminus(\mathrm{cl}D_{2})$, $\tilde{V}_{k}=D_{k}\setminus(\mathrm{cl}D_{k+2})$,
$k=1,2,\ldots$, 
\[
D_{i}=\left(\Phi_{\pi\pi}\right)^{-1}\left(\left(-\frac{1}{i},\frac{1}{i}\right)\times\left(-\frac{1}{i},\frac{1}{i}\right)\setminus\{(0,0)\}\right),\quad i=1,2,\ldots,
\]

According to the proof of Theorem \ref{thm:TopNonOrient} (Lee \cite{Lee}; cf. Section 4), let $\{ \psi_j \}$ be a~smooth partition of unity, subordinate to this cover, and $\theta_j$ be a~$2$-form compactly supported in $V_j\cap V_{j+1}$ for each $j$ (obtained with the help of a~smooth bump function). For every $j$, we put $\omega_j=\psi_j\omega$, $c_1=\int_{V_1}\omega_1$, $c_{j+1}=\int_{V_{j+1}}\left(\omega_{j+1}+c_j\theta_j\right)$. Then there exist a~compactly supported $1$-form $\eta_1$ such that $d\eta_1=\omega_1-c_1\theta_1$, and for each $j$, a~compactly supported $1$-form $\eta_j$ such that $d\eta_j=\omega_{j+1}+c_j\theta_j - c_{j+1}\theta_j$. $1$-form $\eta=\sum_j\eta_j$ is then a~solution of $\omega=d\eta$ on $T_P$.

Note that for constant functions $a$, $b$, and $c$, we get a~global Lagrangian (\ref{Last}), where
\begin{align*}
 & \eta = - r\left(R+r\cos\vartheta\right)\cos\vartheta\left(b\cos\varphi + c\sin\varphi \right) d\vartheta - r\left( aR\cos\vartheta + \frac{1}{4}ar\cos 2\vartheta \right) d\varphi.
\end{align*}

\end{example}

\end{document}